\documentclass[reqno, 12pt,a4letter]{amsart}
\usepackage{amsmath,amsxtra,amssymb,latexsym, amscd,amsthm}
\usepackage[utf8]{inputenc}
\usepackage[mathscr]{euscript}
\usepackage{mathrsfs}
\usepackage[english]{babel}
\usepackage{color}
\usepackage{cite}

\newtheorem {theorem}{Theorem}[section]
\newtheorem {corollary}{Corollary}[section]
\newtheorem {lemma}{Lemma}[section]

\newtheorem {definition}{Definition}[section]
\newtheorem {remark}{Remark}[section]

\def\ar{a\kern-.370em\raise.16ex\hbox{\char95\kern-0.53ex\char'47}\kern.05em}
\def\ees{{\accent"5E e}\kern-.385em\raise.2ex\hbox{\char'23}\kern-.08em}
\def\eex{{\accent"5E e}\kern-.470em\raise.3ex\hbox{\char'176}}
\def\AR{A\kern-.46em\raise.80ex\hbox{\char95\kern-0.53ex\char'47}\kern.13em}
\def\EES{{\accent"5E E}\kern-.5em\raise.8ex\hbox{\char'23 }}
\def\EEX{{\accent"5E E}\kern-.60em\raise.9ex\hbox{\char'176}\kern.1em}
\def\ow{o\kern-.42em\raise.82ex\hbox{
  \vrule width .12em height .0ex depth .075ex \kern-0.16em \char'56}\kern-.07em}
\def\OW{O\kern-.460em\raise1.36ex\hbox{
\vrule width .13em height .0ex depth .075ex \kern-0.16em \char'56}\kern-.07em}
\def\UW{U\kern-.42em\raise1.36ex\hbox{
\vrule width .13em height .0ex depth .075ex \kern-0.16em \char'56}\kern-.07em}
\def\DD{D\kern-.7em\raise0.4ex\hbox{\char '55}\kern.33em}

\title[Bifurcation sets and global monodromies]{Bifurcation sets and global monodromies of Newton non-degenerate polynomials on algebraic sets}

\author{TAT THANG NGUYEN$^{\dag}$}
\address{$^\dag$Institute of Mathematics, 18, Hoang Quoc Viet Road, Cau Giay District 10307, Hanoi, Vietnam}
\email{ntthang@math.ac.vn}

\author{PH\'{U}-PH\'{A}T PH\d{A}M$^\ddag$}
\address{$^{\ddag}$Department of Mathematics, University of Dalat, 1 Phu Dong Thien Vuong, Dalat, Vietnam}
\email{phatpham.pr13@gmail.com}

\author{TI\EES N-S\OW n PH\d{A}M$^{*}$}
\address{$^*$Department of Mathematics, University of Dalat, 1 Phu Dong Thien Vuong, Dalat, Vietnam}
\email{sonpt@dlu.edu.vn}
              
\thanks{The first author is partially supported by Vietnam National Foundation for Science and Technology Development (NAFOSTED) under Grant No.
101.04-2017.12. The second author and the third author are partially supported by Vietnam National Foundation for Science and Technology Development (NAFOSTED)  under Grant No. 101.04-2016.05.}

\thanks{$^{*}$Corresponding author}



\date{ \today}

\begin{document}
\maketitle

\begin{abstract}
Let $S\subset \mathbb{C}^n$ be a non-singular algebraic set and $f \colon \mathbb{C}^n \to \mathbb{C}$ be a polynomial function. It is well-known that the restriction $f|_S \colon S \to \mathbb{C}$ of $f$ on $S$ is a locally trivial fibration outside a finite set $B(f|_S) \subset \mathbb{C}.$ In this paper, we give an explicit description of a finite set $T_\infty(f|_S) \subset \mathbb{C}$ such that $B(f|_S) \subset K_0(f|_S) \cup T_\infty(f|_S),$ where $K_0(f|_S)$ denotes the set of critical values of the $f|_S.$ Furthermore, $T_\infty(f|_S)$ is contained in the set of critical values of certain polynomial functions provided that the $f|_S$ is Newton non-degenerate at infinity. Using these facts, we show that if $\{f_t\}_{t \in [0, 1]}$ is a family of polynomials such that the Newton polyhedron at infinity of $f_t$ is independent of $t$ and the $f_t|_S$ is Newton non-degenerate at infinity, then the global monodromies of the $f_t|_S$ are all isomorphic.
\end{abstract}

\section{Introduction}
Let $S\subset \mathbb{C}^n$ be a non-singular algebraic set and let $f \colon \mathbb{C}^n \rightarrow  \mathbb{C}$ be a polynomial function. In the seventies Thom \cite{Thom1969},  Varchenko \cite{Varchenko1972}, Verdier \cite{Verdier1976} and Wallace \cite{Wallace1971} proved that there exists a finite set $B \subset  \mathbb{C}$ such that the restriction map
\begin{eqnarray*}
f \colon S \setminus f^{-1}(B) \rightarrow \mathbb{C} \setminus B
\end{eqnarray*}
is a locally trivial $C^\infty$-fibration. We call the smallest such $B$ the {\em bifurcation set} of the restriction of $f$ on $S$ and we denote it by $B(f|_S).$  This fibration permits us to introduce the {\em global monodromy} of $f|_S.$ Namely, for $r > \max \{ |c|  \ : \ c \in B(f|_S)\}$ and $\mathbb{S}^1_r := \{c \in \mathbb{C} \ : \ |c| = r\},$ this is the restriction map
\begin{eqnarray*}
f \colon S \cap f^{-1}(\mathbb{S}^1_r) \rightarrow \mathbb{S}^1_r.
\end{eqnarray*}

The problem of studying the bifurcation set and global monodromy of polynomial functions has been extensively studied in several papers: for the case $S = \mathbb{C}^n$ we refer the reader to \cite{Artal2000, Bodin2003, Bodin2004, Broughton1988, Dimca2001, Durfee1998, HaHV1984, HaHV1989-1, HaHV1989-2, HaHV1990, HaHV1991-2, HaHV1996, HaHV1997, Kurdyka2000, Nemethi1990, Nemethi1992, Neumann2000, Phamts2008, Phamts2010, Parusinski1995, Rabier1997, Siersma1995, Siersma1998, Tibar1997}, etc., and for the general case to \cite{HaHV2008-2, Jelonek2004, Jelonek2005}.

Since the restriction $f|_S \colon S \rightarrow  \mathbb{C}$ is not proper, the bifurcation set $B(f|_S)$ of $f|_S$ contains not only the set $K_0(f|_S)$ of critical values of $f|_S,$ but also other values due to the asymptotical ``bad'' behaviour at infinity. To control the set $B(f|_S),$ we use the set $T_\infty(f|_S)$ of {\em tangency values at infinity} of the $f|_S$ (see the definition in Section~\ref{Section3}). It will be shown in Section~\ref{Section3} that 
$T_\infty(f|_S)$ is a finite set and that $B(f|_S) \subset K_0(f|_S) \cup T_\infty(f|_S),$ which means that $f|_S$ is a locally trivial fibration over the complement of $K_0(f|_S) \cup T_\infty(f|_S).$ Furthermore, the set $T_\infty(f|_S)$ is contained in the set of critical values of certain polynomial functions provided that the restriction $f|_S$ is Newton non-degenerate at infinity. These results generalize those given in \cite{Nemethi1990}; for related results we refer the reader to \cite{Bodin2004, Chen2012, Chen2014, Ishikawa2002, Jelonek2004, Jelonek2005, Kurdyka2000, Nguyen2013, Zaharia1996}.

In Section~\ref{Section4}, using the results mentioned above, we will prove a stability theorem, which states that if $\{f_t\}_{t \in [0, 1]}$ is a family of polynomial functions on $\mathbb{C}^n$ such that the Newton polyhedron at infinity of $f_t$ is independent of $t$ and the restriction $f_t|_S$ is Newton non-degenerate at infinity, then the global monodromies of the  $f_t|_S$ are all isomorphic. This generalizes \cite[Theorem~17]{Nemethi1992} and \cite[Theorem~1.1]{Phamts2010}, where the case $S = \mathbb{C}^n$ was studied.

\section{Notations and Definitions} \label{Preliminary}

In this section we present some notations and definitions, which are used throughout this paper. 

\subsection{Notations}

We suppose $1 \leqslant  n \in \mathbb{N}$ and abbreviate $(x_1, \ldots, x_n)$ by $x.$  Let $\mathbb{K} := \mathbb{R}$ or $\mathbb{C}.$ The inner product (resp., norm) on $\mathbb{K}^n$  is denoted by $\langle x, y \rangle$ for any $x, y \in \mathbb{K}^n$ (resp., $\| x \| := \sqrt{\langle x, x \rangle}$ for any $x \in \mathbb{K}^n$). The real part and complex conjugate of a complex number $c \in \mathbb{C}$ are denoted by $\Re c$ and $\overline{c},$ respectively.

For each $r > 0,$ we will write $D_r := \{c \in \mathbb{C} \ : \ |c| < r\}$ for the open disc and write $\mathbb{S}^{2n - 1}_r :=
\{x \in \mathbb{C}^n \ : \ \|x\| = r\}$ for the sphere.

Given nonempty sets $I \subset \{1, \ldots, n\}$ and $A \subset \mathbb{K}^n,$ we define
$$A^I := \{x \in A \ : \ x_i = 0 \textrm{ for all } i \not \in I\}.$$

Let $\mathbb{C}^* := \mathbb{C} \setminus \{0\}$ and we denote by $\mathbb{Z}_+$ the set of non-negative integer numbers. If $\alpha = (\alpha_1, \ldots, \alpha_n) \in \mathbb{Z}_+^n,$ we denote by $x^\alpha$ the monomial $x_1^{\alpha_1} \cdots x_n^{\alpha_n}.$

The gradient of a polynomial function $f \colon \mathbb{C}^n \to \mathbb{C}$ is denoted by $\nabla f$ as usual, i.e., 
$$\nabla f(x) := \left(\overline{\frac{\partial f}{\partial x_1}(x), \ldots, \frac{\partial f}{\partial x_n}(x)} \right),$$
so the chain rule may expressed by the inner product $\partial f/\partial {\mathbf{v}} = \langle \mathbf{v}, \nabla f \rangle.$

\subsection{Newton polyhedra and non-degeneracy conditions}

Let $f \colon \mathbb{C}^n \to \mathbb{C}$ be a polynomial function. Suppose that $f$ is written as $f = \sum_{\alpha} a_\alpha x^\alpha.$ Then
the support of $f,$ denoted by $\mathrm{supp}(f),$ is defined as the set of those $\alpha  \in \mathbb{Z}_+^n$ such that $a_\alpha \ne 0.$ 
The {\em Newton polyhedron\footnote{Note that we do not include the origin in the definition of the Newton polyhedron $\Gamma(f).$} (at infinity)}  of $f$, denoted by $\Gamma(f),$ is defined as the convex hull in $\mathbb{R}^n$ of the set $\mathrm{supp}(f).$ 
The polynomial $f$ is said to be {\em convenient} if $\Gamma(f)$ intersects each coordinate axis in a point different from the origin $0$ in $\mathbb{R}^n.$ 
For each (closed) face $\Delta$ of $\Gamma(f),$  we will denote by $f_\Delta$ the polynomial $\sum_{\alpha \in \Delta} a_\alpha x^\alpha;$ if $\Delta \cap \mathrm{supp}(f) = \emptyset$ we let $f_\Delta := 0.$

Given a nonzero vector $q \in \mathbb{R}^n,$ we define
\begin{eqnarray*}
d(q, \Gamma(f)) &:=& \min \{\langle q, \alpha \rangle \ : \ \alpha \in \Gamma(f)\}, \\
\Delta(q, \Gamma(f)) &:=& \{\alpha \in \Gamma(f) \ : \ \langle q, \alpha \rangle = d(q, \Gamma(f)) \}.
\end{eqnarray*}
By definition, for each nonzero vector $q \in \mathbb{R}^n,$ $\Delta(q, \Gamma(f)) $ is a closed face of $\Gamma(f).$ Conversely, if $\Delta$ is a closed face of $\Gamma(f)$ then there exists a nonzero vector\footnote{Since $\Gamma(f)$ is an integer polyhedron, we can assume that all the coordinates of $q$ are rational numbers.} $q \in \mathbb{R}^n$ such that $\Delta = \Delta(q, \Gamma(f)).$ 

\begin{remark}{\rm
The following statements follow immediately from definitions:

(i) For each nonempty subset $I$ of $\{1, \ldots, n\},$ if the restriction of $f$ on $\mathbb{C}^I$ is not identically zero, then $\Gamma(f) \cap \mathbb{R}^I = \Gamma(f|_{\mathbb{C}^I}).$

(ii) Let $\Delta := \Delta(q, \Gamma(f))$ for some nonzero vector $q := (q_1, \ldots, q_n) \in \mathbb{R}^n.$ By definition, $f_\Delta = \sum_{\alpha \in \Delta} a_\alpha x^\alpha$ is a weighted homogeneous polynomial of type $(q, d := d(q, \Gamma(f))),$ i.e., we have for all $t > 0$ and all $x \in \mathbb{C}^n,$
$$f_\Delta(t^{q_1} x_1,  \ldots, t^{q_n} x_n) = t^d f_\Delta(x_1, \ldots, x_n).$$
This implies the Euler relation
\begin{equation*}
\sum_{i = 1}^n q_i x_i \frac{\partial f_\Delta}{\partial x_i}(x) = d f_\Delta(x).
\end{equation*}
In particular, if $d \ne 0$ and $\nabla f_\Delta(x) = 0,$ then $f_\Delta(x) = 0.$
}\end{remark}

For the rest of this section, let $g_1, \ldots, g_p \colon \mathbb{C}^n \rightarrow \mathbb{C}$ be polynomial functions and set
$$S := \{x \in \mathbb{C}^n \ : \ g_1(x)  = 0, \ldots, g_p(x)  = 0\}.$$
The following definition of non-degeneracy is inspired from the work of Kouchnirenko~\cite{Kouchnirenko1976}, where the case $S = \mathbb{C}^n$ was considered.

\begin{definition}\label{Definition21}{\rm
We say that the restriction of $f$ on $S$ is {\em Newton non-degenerate at infinity} if, and only if, for every nonempty set $I \subset \{1, \ldots, n\}$ with $f|_{\mathbb{C}^I} \not \equiv 0,$ for every (possibly empty) set $J \subset \{j \in \{1, \ldots, p\} \ : \ g_j|_{\mathbb{C}^I} \not \equiv 0\},$ and for every vector $q \in \mathbb{R}^n$ with $\min_{i \in I} q_i < 0,$ the following conditions hold:
\begin{itemize}
\item[(i)] the set 
\begin{eqnarray*}
\{x \in \mathbb{C}^{*n} \ : \ g_{j, \Delta_j} (x) = 0 \textrm{ for } j \in J \}
\end{eqnarray*}
is a reduced smooth complete intersection variety in the torus $\mathbb{C}^{*n},$ i.e., the system of gradient vectors $\nabla g_{j, \Delta_j}(x)$ for $j \in J$ is $\mathbb{C}$-linearly independent on this variety;

\item[(ii)] if $d(q, \Gamma(f|_{\mathbb{C}^I})) < 0,$ then the set 
\begin{eqnarray*}
\{x \in \mathbb{C}^{*n} \ : \ f_{\Delta_0}(x) = 0 \textrm{ and } g_{j, \Delta_j} (x) = 0 \textrm{ for } j \in J \}
\end{eqnarray*}
is a reduced smooth complete intersection variety in the torus $\mathbb{C}^{*n};$
\end{itemize}
where $\Delta_0 := \Delta(q, \Gamma(f|_{\mathbb{C}^I}))$ and $\Delta_j := \Delta(q, \Gamma(g_j|_{\mathbb{C}^I}))$ for $j \in J.$
}\end{definition}

Finally, following \cite{Nemethi1990}, we introduce a set, which plays an important role in the sequel. Namely, let $\Sigma_\infty(f|_S)$ denote the set of all values $c \in \mathbb{C}$ for which there exist 
a nonempty set $I \subset \{1, \ldots, n\}$ with $f|_{\mathbb{C}^I} \not \equiv 0,$ 
a (possibly empty) set $J \subset \{j \in \{1, \ldots, p\} \ : \ g_j|_{\mathbb{C}^I} \not \equiv 0\},$ 
a vector $q \in \mathbb{R}^n$ with $\min_{i \in I} q_i < 0$ and $d(q, \Gamma(f|_{\mathbb{C}^I})) = 0,$ 
a point $x \in \mathbb{C}^{*I},$ and scalars $\lambda_j \in \mathbb{C}$ for $j \in J,$ such that
the following conditions hold:
\begin{eqnarray*}
&& c \ = \ f_{\Delta_0}(x), \\
&& g_{j, \Delta_j} (x) \ = \ 0 \textrm{ for } j \in J, \\
&& \nabla f_{\Delta_0}(x) + \sum_{j \in J} \lambda_j \nabla g_{j, \Delta_j} (x) \ = \ 0,
\end{eqnarray*}
where $\Delta_0 := \Delta(q, \Gamma(f|_{\mathbb{C}^I}))$ and $\Delta_j := \Delta(q, \Gamma(g_j|_{\mathbb{C}^I}))$ for $j \in J.$ 

We observe that the above value $c \in \Sigma_\infty(f|_S)$ is indeed a critical value of the restriction of the polynomial $f_{\Delta_0}$ on the variety
$$\{x \in \mathbb{C}^{*I} \ : \ g_{j, \Delta_j} (x) = 0 \quad \textrm{ for } \quad j \in J\}.$$ 
Hence, by the Bertini--Sard theorem, $\Sigma_\infty(f|_S)$ is a finite set provided that the restriction $f|_S$ is Newton  non-degenerate at infinity.

\section{The bifurcation set of a polynomial function} \label{Section3}

From now on, let $g_1, \ldots, g_p \colon \mathbb{C}^n \rightarrow \mathbb{C}$ be polynomial functions such that the algebraic set
$$S := \{x \in \mathbb{C}^n \ : \ g_1(x)  = 0, \ldots, g_p(x)  = 0\}$$
is a reduced smooth complete intersection variety, i.e., the system of gradient vectors 
$$\nabla g_1(x), \ldots, \nabla g_p(x)$$
is $\mathbb{C}$-linearly independent for all $x \in S.$ 

\begin{lemma}\label{Lemma31}
There exists a real number $R_0 > 0$ such that for all $R \geqslant R_0,$ the set $S$ intersects transversally with the sphere $\mathbb{S}^{2n-1}_R$.
\end{lemma}

\begin{proof}
We argue by contradiction. Suppose that there exist sequences $\{x^k\}_{k\in \mathbb{N}} \subset \mathbb{C}^n$ and $\{\lambda_j^k\}_{k\in \mathbb{N}} \subset \mathbb{C}, j=1,\ldots,p+1,$ such that
\begin{enumerate}
\item[(a1)] $\Vert x^k \Vert \rightarrow \infty$ as $k\rightarrow \infty;$
\item[(a2)] $g_j(x^k) = 0$ for all $j=1,\ldots,p,$ and all $k \in \mathbb{N};$
\item[(a3)] $\sum_{j=1}^p\lambda_j^k\nabla g_j(x^k) = \lambda_{p+1}^k {x^k};$
\item[(a4)] The numbers $\lambda_j^k, j = 1, \ldots, p + 1,$ are not all zero for all $k\in \mathbb{N}.$
\end{enumerate} 
By the Curve Selection Lemma at infinity (see \cite{Nemethi1992} or \cite{HaHV2017}), there exist analytic curves
\begin{eqnarray*}
\phi \colon (0,\epsilon) \rightarrow \mathbb{C}^n \quad \textrm{ and } \quad  \lambda_j \colon (0,\epsilon) \rightarrow \mathbb{C}, \ j = 1,\ldots,p+1,
\end{eqnarray*}
such that
\begin{enumerate}
\item[(a5)] $\Vert \phi(s) \Vert \rightarrow \infty$ as $ s\rightarrow 0;$
\item[(a6)] $g_j(\phi(s)) = 0$ for all $j=1,\ldots,p$, and all $s \in (0,\epsilon);$
\item[(a7)] $\sum_{j=1}^p\lambda_j(s)\nabla g_j(\phi(s)) = \lambda_{p+1}(s) {\phi(s)}\textrm{ for } s\in (0,\epsilon);$
\item[(a8)] $\lambda_j(s), j = 1, \ldots, p + 1,$ are not all zero for $s\in(0,\epsilon).$
\end{enumerate}

We have
\begin{eqnarray*}
\left\langle \frac{{d \phi(s)}}{ds}, \sum_{j=1}^p\lambda_j(s)\nabla g_j(\phi(s)) \right\rangle 
&=& \sum_{j=1}^p \overline{\lambda_j(s)} \left\langle \frac{{d \phi(s)}}{ds}, \nabla g_j(\phi(s)) \right\rangle\\
&=& \sum_{j = 1}^p \overline{\lambda_j(s)} \frac{d}{ds} (g_j \circ \phi)(s) \\
&=& 0.
\end{eqnarray*}
Combined with the condition (a7), this implies that
\begin{eqnarray*}
0 & = & \overline{\lambda_{p+1}(s)} \Re \left\langle \frac{{d\phi(s)}}{ds}, {\phi(s)} \right\rangle\ =\ \overline{\lambda_{p+1}(s)} \frac{d\Vert\phi(s)\Vert^2}{2ds}.
\end{eqnarray*}
But $\lambda_{p+1}\not\equiv 0,$ which follows from the non-singularity of $S$ and the condition (a7). Hence,
\begin{eqnarray*}
\frac{d\Vert\phi(s)\Vert^2}{ds} &=& 0
\end{eqnarray*}
for all $s > 0$ small enough, which contradicts the condition (a5).
\end{proof}

For the rest of this section, let $f \colon \mathbb{C}^n \rightarrow \mathbb{C}$ be a polynomial function. It is well known that 
the bifurcation set $B(f|_S)$ of the restriction $f|_S \colon S \rightarrow \mathbb{C}$ contains the set $K_0(f|_S).$ Recall that we write $K_0(f|_S)$ for {\em the set of critical values} of the restriction of $f$ on $S,$ i.e., 
\begin{eqnarray*}
K_0(f|_S) := \{c\in \mathbb{C} & : &  \exists x \in S, \exists \lambda_j \in \mathbb{C}, j = 1, \ldots, p,  \textrm{ such that } \\
&& f(x) = c \quad \textrm{ and } \quad \nabla f(x) + \sum_{j = 1}^p \lambda_j \nabla g_j(x) = 0\}.
\end{eqnarray*}
By the Bertini--Sard theorem, $K_0(f|_S)$ is a finite set.

Before formulating our first theorem, we also need the following concept (see also \cite[Chapter~2]{HaHV2017}). 

\begin{definition}{\rm
By the {\em set of tangency values at infinity} of the $f|_S$ we mean the set
\begin{eqnarray*}
T_\infty(f|_S) &:= & \{c \in \mathbb{C} \ :  \ \exists \{x^k\} \subset S, \exists \{\lambda^k_j\}\subset \mathbb{C}, j = 1, \ldots, p + 1, \|x^k\|\to \infty, f(x^k)\to c, \\
& & \qquad \qquad \ 
\nabla f(x^k) + \sum_{j=1}^p \lambda_j^k \nabla g_j(x^k)= \lambda_{p+1}^k {x^k} \ \textrm{ for all } \ k \in \mathbb{N} \}.
\end{eqnarray*}
}\end{definition}

Notice that for $S = \mathbb{C}^n$ the set $T_\infty(f|_S) $ coincides to the set $S_f$ defined by Nemethi and Zaharia \cite{Nemethi1990}.

\begin{theorem}\label{Theorem31}
$T_\infty(f|_S)$ is a finite set and the following inclusion holds
\begin{eqnarray}\label{EQ01}
B(f|_S) & \subset & K_0(f|_S) \cup T_\infty(f|_S).
\end{eqnarray}
\end{theorem}

\begin{proof}
In order to prove the set $T_\infty(f|_S)$ is finite, we use the {\em set of asymptotic critical values at infinity} of $f|_S$ (see \cite{Jelonek2004, Jelonek2005, Kurdyka2000, Rabier1997}): 
\begin{eqnarray*}
K_{\infty}(f, S) &:= & \{c \in \mathbb{C} \ :  \ \exists \{x^k\}\subset S, \|x^k\|\to \infty, f(x^k)\to c, \mbox{and } \\
& & \qquad \qquad \  \|x^k\|\nu(x^k)\to 0  \, \mbox{ as } \, k\to\infty\},
\end{eqnarray*}
where $\nu \colon\mathbb{C}^n\to\mathbb{R}$ is the {\em Rabier function} defined by
$$\nu(x) := \inf \left \{ \left\| \nabla f(x) + \sum_{j = 1}^p \lambda_j \nabla g_j(x) \right\| \ : \ \lambda_j \in \mathbb{C}, j = 1, \ldots, p\right \}.$$
We will show that
\begin{eqnarray}\label{EQ02}
T_\infty(f|_S) & \subset & K_\infty(f, S).
\end{eqnarray}
This, of course, implies immediately that $T_\infty(f|_S)$ is a finite set because we know from \cite[Theorem~3.3]{Jelonek2005} that $K_\infty(f|_S)$ is a finite set.

In order to prove the inclusion~\eqref{EQ02}, take any $c \in T_{\infty}(f|_S).$ By definition, there exist sequences $\{x^k\}_{k \in \mathbb{N}} \subset \mathbb{C}^n$ and $\{\lambda^k_j\}_{k \in \mathbb{N}} \subset \mathbb{C}, j = 1, \ldots, p + 1,$ such that 
\begin{enumerate}
\item[(a1)] $\Vert x^k \Vert \rightarrow \infty$ as $k\rightarrow \infty$;
\item[(a2)] $f(x^k) \rightarrow c$ as $k \rightarrow \infty;$
\item[(a3)] $g_j(x^k) = 0$ for all $j=1,\ldots,p,$ and all $k \in \mathbb{N};$
\item[(a4)] $\nabla f(x^k) + \sum_{j=1}^p \lambda_j^k \nabla g_j(x^k)= \lambda_{p+1}^k {x^k}$ for all $k \in \mathbb{N}.$
\end{enumerate} 
By the Curve Selection Lemma at infinity (see \cite{Nemethi1992} or \cite{HaHV2017}), there exist analytic curves
\begin{eqnarray*}
\phi \colon (0,\epsilon) \rightarrow \mathbb{C}^n \quad \textrm{ and } \quad  \lambda_j \colon (0,\epsilon) \rightarrow \mathbb{C}, \ j =  1, \ldots, p + 1,
\end{eqnarray*}
such that
\begin{enumerate}
\item[(a5)] $\Vert \phi(s) \Vert \rightarrow \infty$ as $ s\rightarrow 0;$
\item[(a6)] $f(\phi(s)) \rightarrow c$ as $s\rightarrow 0$;
\item[(a7)] $g_j(\phi(s)) = 0$ for all $j=1,\ldots,p,$ and all $s \in (0,\epsilon);$
\item[(a8)] $\nabla f(\phi(s)) + \sum_{j=1}^p \lambda_j(s) \nabla g_j(\phi(s)) = \lambda_{p+1}(s) {\phi(s)}\textrm{ for } s\in (0,\epsilon).$
\end{enumerate}
If $\lambda_{p + 1} \equiv 0,$ then it is clear that $c \in K_\infty(f, S)$ and there is nothing to prove. So we may assume that $\lambda_{p + 1}$ is not identically zero.
It follows from (a7) and (a8) that
\begin{eqnarray*}
0 \ \not \equiv \  \frac{d \|\phi(s)\|^2}{2 d s}
&=& \Re \left \langle \frac{{d \phi(s)}}{d s}, {\phi(s)} \right \rangle \\
&=& \Re \left \langle \frac{{d \phi(s)}}{d s}, \frac{1}{{\lambda_{p + 1}(s)}} \left [\nabla f(\phi(s)) + \sum_{j=1}^p \lambda_j(s) \nabla g_j(\phi(s)) \right] \right \rangle \\
&=& \Re \frac{1}{\overline{\lambda_{p + 1}(s)}} \left [\frac{d}{ds}(f\circ\phi)(s) + \sum_{j=1}^p \overline{\lambda_j(s)} \frac{d}{ds}(g_j\circ\phi)(s) \right] \\ 
&=& \Re \frac{1}{\overline{\lambda_{p + 1}(s)}} \left[\frac{d}{ds}(f\circ\phi)(s) \right].
\end{eqnarray*}
In particular, $f\circ\phi \not \equiv c.$ 

On the other hand, we may write
\begin{eqnarray*}
\|\phi(s)\| & = & a s^{\alpha} + \textrm{higher-order terms in } s,\\
f (\phi(s)) & = & c + b s^{\beta} + \textrm{higher-order terms in } s,
\end{eqnarray*}
where $a \ne 0, b \ne 0$ and $\alpha, \beta \in \mathbb{Q}.$ By the conditions (a5) and (a6) respectively, then $\alpha < 0$ and $\beta  > 0.$ Therefore, we have asymptotically as $s \to 0^+,$ 
\begin{eqnarray*}
|\lambda_{p + 1}(s)| &\simeq & s^{\beta - 2 \alpha}.
\end{eqnarray*}
It turns out from (a8) that
\begin{eqnarray*}
\|\phi(s)\| \| \nabla f(\phi(s)) + \sum_{j=1}^p \lambda_j(s) \nabla g_j(\phi(s)) \| & \simeq & s^{\beta}  \quad \textrm{ as } \ s \to 0^+,
\end{eqnarray*}
which yields $c \in K_\infty(f, S).$ Hence the inclusion~\eqref{EQ02} holds. 

For the proof of the inclusion~\eqref{EQ01} we fix $c^* \in \mathbb{C} \setminus (K_0(f|_S) \cup T_\infty(f|_S))$ and $D$ a small open disc centered at $c^*,$ with the closure $\overline{D} \subset \mathbb{C} \setminus (K_0(f|_S) \cup T_\infty(f|_S)).$ Then it is not hard to see that there exists a real number $R_0 > 0$ such that
for all $c \in D$ and all $R \geqslant  R_0,$ the fiber $(f|_S)^{-1}(c)$ is non-singular and intersects transversally with the sphere $\mathbb{S}^{2n - 1}_R$ (this is possible if $D$ is small enough). By continuity, there exists an open neighbourhood $U$ of $(f|_S)^{-1}(D) \cap \{x \in \mathbb{C}^n \ : \ \|x\| \ge R_0\}$ in $\mathbb{C}^n$ such that the vectors $\nabla f(x), \nabla g_1(x), \ldots, \nabla g_p(x),$ and ${x}$ are $\mathbb{C}$-linearly independent for all $x \in U.$ Therefore, we can find a smooth vector field $\mathbf{v}_1$ on $U$ satisfying the following conditions
\begin{itemize}
\item[(a1)] $\langle \mathbf{v}_1(x), {\nabla f(x)} \rangle = 1;$
\item[(a2)] $\langle \mathbf{v}_1(x), {\nabla g_j(x)} \rangle = 0$ for $j = 1, \ldots, p;$
\item[(a3)] $\langle \mathbf{v}_1(x), x \rangle = 0.$
\end{itemize}
(We can construct such a vector field locally, then extend it over $U$ by a smooth partition of unity.) 

We now fix $\epsilon > 0.$ Since $D \cap K_0(f, S) = \emptyset,$ the vectors $\nabla f(x), \nabla g_1(x), \ldots, \nabla g_p(x)$ are $\mathbb{C}$-linearly independent for all $x$ belonging to some open neighbourhood $V$ of $(f|_S)^{-1}(D) \cap \{x \in \mathbb{C}^n \ : \ \|x\| \le R_0 + \epsilon\}$ in $\mathbb{C}^n.$ Consequently, there exists a smooth vector field $\mathbf{v}_2$ on $V$ such that the following conditions hold
\begin{itemize}
\item[(a4)] $\langle \mathbf{v}_2(x), {\nabla f(x)} \rangle = 1;$
\item[(a5)] $\langle \mathbf{v}_2(x), {\nabla g_j(x)} \rangle = 0$ for $j = 1, \ldots, p.$
\end{itemize}
(We can construct such a vector field locally, then extend it over $V$ by a smooth partition of unity.) 

Next, we fix a partition of unity $\theta_1$ and $\theta_2$ subordinated to the covering
$$\left \{x \in U \ : \ \|x\| > R_0 + \frac{\epsilon}{3} \right\} \quad \textrm{ and  } \quad \left  \{x \in V \ : \ \|x\| < R_0 + \frac{2\epsilon}{3} \right\}$$
of $(f|S)^{-1}(D),$ and define the smooth vector field $\mathbf{v}$ on $(f|S)^{-1}(D)$ by 
$$\mathbf{v} := \theta_1 \mathbf{v}_1 + \theta_2 \mathbf{v}_2.$$
Then we can see that the following conditions hold:
\begin{itemize}
\item[(a6)] $\langle \mathbf{v}(x), {\nabla f(x)} \rangle = 1;$
\item[(a7)] $\langle \mathbf{v}(x), {\nabla g_j(x)} \rangle = 0$ for $j = 1, \ldots, p;$
\item[(a8)] $\langle \mathbf{v}(x), x \rangle = 0$ provided that $\|x\| \ge R_0 + \epsilon.$
\end{itemize}
Finally, integrating the vector field $\mathbf{v}$ we have that the restriction $f \colon (f|_S)^{-1}(D) \rightarrow D$ is a trivial $C^\infty$-fibration, which means that $c^* \not \in B(f|_S).$
\end{proof}

\begin{remark}{\rm
(i) The inclusion~\eqref{EQ01} provides an extension to {\em algebraic sets} of Theorem~1 in \cite{Nemethi1990}, where the case $S = \mathbb{C}^n$ was studied.

(ii) The inclusions~\eqref{EQ01} and \eqref{EQ02} may be strict in general, see \cite{Paunescu1997, Paunescu2000} and \cite{HaHV2008-1}.

(iii) The proof of Theorem~\ref{Theorem31} also implies the following inclusion, which was proved in \cite{Rabier1997, Jelonek2004, Jelonek2005},
\begin{eqnarray*}
B(f|_S) & \subset & K_0(f|_S) \cup K_\infty(f|_S).
\end{eqnarray*}

(iv) A straightforward modification shows that Lemma~\ref{Lemma31} and Theorem~\ref{Theorem31} still hold in the case where $S$ does not have the explicit form as it was assumed; in fact, it suffices to suppose that $S$ is a non-singular constructive subset of $\mathbb{C}^n.$ As we shall not use this ``improve'' statement, we leave the proof as an exercise.
}\end{remark}

Under the non-degeneracy condition of Definition~\ref{Definition21}, we obtain the following bound of tangency values at infinity of $f|_S$ in terms of critical values of certain polynomial functions.

\begin{theorem}\label{Theorem32}
Assume that the restriction $f|_S$ of $f$ on $S$ is Newton non-degenerate at infinity. Then
\begin{eqnarray*}
T_{\infty}(f|_S) &\subset& \Sigma_{\infty}(f|_S) \cup K_0(f|_S) \cup \{0\}.
\end{eqnarray*}
Moreover, if the polynomial $f  \colon \mathbb{C}^n \rightarrow \mathbb{C}$ is convenient, then $T_{\infty}(f|_S) = \emptyset.$
\end{theorem}

\begin{proof}
For convenience we will write $g_0$ instead of $f.$

Take arbitrary $c \in T_{\infty}(g_0|_S) \setminus (K_0(g_0|_S) \cup \{0\}).$ We will show that $c \in \Sigma_{\infty}(f|_S).$ Indeed, by definition, there exist sequences $\{x^k\}_{k \in \mathbb{N}} \subset \mathbb{C}^n$ and $\{\lambda^k_j\}_{k \in \mathbb{N}} \subset \mathbb{C}, j = 1, \ldots, p + 1,$ such that 
\begin{enumerate}
\item[(a1)] $\Vert x^k \Vert \rightarrow \infty$ as $k\rightarrow \infty$;
\item[(a2)] $g_0(x^k) \rightarrow c$ as $k \rightarrow \infty;$
\item[(a3)] $g_j(x^k) = 0$ for all $j=1,\ldots,p,$ and all $k \in \mathbb{N};$
\item[(a4)] $\nabla g_0(x^k) + \sum_{j=1}^p \lambda_j^k \nabla g_j(x^k) = \lambda_{p+1}^k {x^k}$ for all $k \in \mathbb{N}.$
\end{enumerate} 
By the Curve Selection Lemma at infinity (see \cite{Nemethi1992} or \cite{HaHV2017}), there exist analytic curves
\begin{eqnarray*}
\phi \colon (0,\epsilon) \rightarrow \mathbb{C}^n \quad \textrm{ and } \quad  \lambda_j \colon (0,\epsilon) \rightarrow \mathbb{C}, \ j =  1, \ldots, p + 1,
\end{eqnarray*}
such that
\begin{enumerate}
\item[(a5)] $\Vert \phi(s) \Vert \rightarrow \infty$ as $ s\rightarrow 0;$
\item[(a6)] $g_0(\phi(s)) \rightarrow c$ as $s\rightarrow 0$;
\item[(a7)] $g_j(\phi(s)) = 0$ for all $j=1,\ldots,p,$ and all $s \in (0,\epsilon);$
\item[(a8)] $\nabla g_0(\phi(s)) + \sum_{j=1}^p \lambda_j(s) \nabla g_j(\phi(s)) = \lambda_{p+1}(s) {\phi(s)}\textrm{ for } s\in (0,\epsilon).$
\end{enumerate}

Put $I:=\{ i \ : \ \phi_i \not\equiv 0 \}$. By the condition (a5), $I \neq \emptyset$. For $i \in I$, we can write the curve $\phi_i$ in terms of parameter, say
$$\phi_i(s) \ =\ x^0_i s^{q_i} + \textrm{higher-order terms in }s,$$
where $x_i^0 \neq 0$ and $q_i \in \mathbb{Q}.$ We have $\min_{i \in I} q_i < 0,$ because of the condition (a5).

If $\lambda_{p+1}\equiv 0,$ then it follows from the conditions (a7) and (a8) that
\begin{eqnarray*}
\frac{d}{ds}(g_0\circ\phi)(s) 
\ = \ \left \langle {\frac{d\phi(s)}{ds}}, \nabla g_0(\phi(s)) \right \rangle & = &  -\sum_{j=1}^p \overline{\lambda_j(s)} \left  \langle {\frac{d\phi (s)}{ds}}, \nabla g_j(\phi(s)) \right \rangle \\ 
& = &  -\sum_{j=1}^p \overline{\lambda_j(s)} \frac{d}{ds}(g_j\circ\phi)(s) \ = \ 0.
\end{eqnarray*}
Consequently, $g_0(\phi(s))=c$ for $ s \in (0,\epsilon),$ and so $c \in K_0(g_0|_S)$, which is  a contradiction. Therefore, $\lambda_{p+1}\not\equiv 0.$ 
Put $J:= \{ j\in \{ 1, \ldots, p\} \ : \ \lambda_j \not\equiv 0 \}.$ For $j \in J \cup \{p+1\}$, we can write
\begin{eqnarray*}
\lambda_j(s) \ =\ c_j s^{m_j} +  \textrm{higher-order terms in }s ,
\end{eqnarray*}
where $c_j \neq 0$ and $m_j \in \mathbb{Q}.$ 

Put $J_1:= \{ j\in \{ 0\} \cup J \ : \ g_j|_{\mathbb{C}^I} \not\equiv 0\}$. The condition (a6) and the assumption that $c\neq 0$ together imply that $0\in J_1,$ and so $J_1\neq \emptyset.$ For each $j \in J_1,$ let $d_j$ be the minimal value of the linear function $\sum_{i \in I}\alpha_i q_i$ on $ \mathbb{R}^I \cap \Gamma(g_j)$ and $\Delta_j$ be the maximal face of $\mathbb{R}^I \cap \Gamma(g_j)$, where this linear function takes its minimum value, respectively. A simple calculation shows that
\begin{eqnarray*}
g_j(\phi(s)) &=& g_{j,\Delta_j}(x^0) s^{d_j} + \textrm{ higher-order terms in }s,
\end{eqnarray*}
where $x^0 := (x^0_1, \ldots, x^0_n)$ with $x^0_i = 1$ for $i \not \in I$ and $g_{j,\Delta_j}$ is the face function associated with $g_j$ and $\Delta_j.$ The condition (a6) and the assumption that $c\neq 0$ together imply that
\begin{eqnarray} \label{EQ03}
d_0 \leqslant 0 \quad \textrm{ and } \quad d_0 g_{0,\Delta_0}(x^0) = 0.
\end{eqnarray}
Furthermore, it follows from the condition (a7) that
\begin{eqnarray} \label{EQ04}
g_{j,\Delta_j}(x^0) & = & 0 \quad \textrm{ for all } \quad j \in J_1 \setminus \{0\}.
\end{eqnarray}

On the other hand, we have for all $i \in I$ and all $j \in J_1,$
\begin{eqnarray*}
\frac{\partial g_j(\phi(s))}{\partial x_i}  &=& \frac{\partial g_{j,\Delta_j}}{\partial x_i} (x^0) s^{d_j - q_i} + \textrm{ higher-order terms in }s.
\end{eqnarray*}
Combined with the condition (a8), this equation implies that for all $i \in I,$
\begin{eqnarray*}
\left( \sum_{j \in J_2} \overline{c_j} \frac{\partial g_{j,\Delta_j}}{\partial x_i}(x^0)\right) s^{\ell-q_i} + \cdots &  = & \overline{c_{p+1}} \overline{ x^0_i}s^{m_{p+1}+q_i} + \cdots,
\end{eqnarray*}
where $c_0 :=1 , m_0 := 0, \ell := \min\{m_j+d_j \ : \ j \in J_1\},$  $J_2:= \{ j \in J_1 \ : \ \ell  = m_j + d_j \},$ and the dots stand for the higher-order terms in $s.$
Clearly, $\ell - q_i \leqslant m_{p+1} + q_i$ for all $i \in I.$ Therefore,
\begin{eqnarray} \label{EQ05}
\ell - m_{p+1} & \leqslant  & 2 \min_{i\in I}q_i  \ < \ 0.
\end{eqnarray}

We next show that the set $I_1:= \{i \in I \ : \ \ell-q_i=m_{p+1}+ q_i\}$ is empty. To see this, we observe that
\begin{eqnarray*}
\sum_{j \in J_2} \overline{c_j} \frac{\partial g_{j,\Delta_j}}{\partial x_i}(x^0) & = & 
\begin{cases}
\overline{c_{p+1}} \overline{x^0_i} & \textrm{ if } i \in I_1, \\
0  & \textrm{ if } i \in I \setminus I_1, \\
0  & \textrm{ if } i \not \in I,
\end{cases}
\end{eqnarray*}
where the last equation holds because for all $i \not \in I$ and all $j \in J_2,$ the polynomial $g_{j, \Delta_j}$ does not depend on the variable $x_i.$
Consequently,
\begin{eqnarray*}
\sum_{i = 1}^n \left(\sum_{j \in J_2} \overline{c_j} \frac{\partial g_{j,\Delta_j}}{\partial x_i}(x^0)\right) x^0_i q_i & = & \nonumber
\sum_{i \in I_1} \left(\sum_{j \in J_2} \overline{c_j} \frac{\partial g_{j,\Delta_j}}{\partial x_i}(x^0)\right) x^0_i q_i \\
& = & \sum_{i\in I_1} \overline{c_{p+1}}|x^0_i|^2 \frac{\ell-m_{p+1}}{2}.
\end{eqnarray*}

On the other hand, by the Euler relation, we have for all $j \in J_2,$
\begin{eqnarray*}
\sum_{i = 1}^n  \frac{\partial g_{j,\Delta_j}}{\partial x_i}(x^0) x^0_i q_i & = &  d_j g_{j,\Delta_j}(x^0).
\end{eqnarray*}
It follows that
\begin{eqnarray*}
\sum_{i = 1}^n \left(\sum_{j \in J_2} \overline{c_j} \frac{\partial g_{j,\Delta_j}}{\partial x_i}(x^0)\right) x^0_i q_i 
& = &  \sum_{j \in J_2} \overline{c_j} \left( \sum_{i = 1}^n  \frac{\partial g_{j,\Delta_j}}{\partial x_i}(x^0) x^0_i q_i \right) \\
& = &  \sum_{j \in J_2} \overline{c_j} d_j g_{j,\Delta_j}(x^0).
\end{eqnarray*}
Therefore,
\begin{eqnarray*}
\sum_{i\in I_1} \overline{c_{p+1}} |x^0_i|^2 \frac{\ell-m_{p+1}}{2} & = & \sum_{j \in J_2} \overline{c_j} d_j g_{j,\Delta_j}(x^0).
\end{eqnarray*}
This, together with \eqref{EQ03}, \eqref{EQ04}, and \eqref{EQ05}, gives $I_1 = \emptyset.$ Thus, 
\begin{eqnarray*}
\sum_{j \in J_2} \overline{c_j} \frac{\partial g_{j,\Delta_j}}{\partial x_i}(x^0) & = & 0 \quad \textrm{ for all } \quad i = 1,\ldots,n.
\end{eqnarray*}
Since the restriction of $g_0$ on $S$ is Newton non-degenerate at infinity, we deduce easily from \eqref{EQ03} and \eqref{EQ04} that $d_0 = 0$ and $0 \in J_2,$ hence that $c = g_{0, \Delta_0}(x^0) \in \Sigma_\infty(g_0|_S).$

Finally, assume that the polynomial $f  \colon \mathbb{C}^n \rightarrow \mathbb{C}$ is convenient. Then $d_0 < 0,$ which is a contradiction. Hence $T_\infty(f|_S) = \emptyset.$
\end{proof}
 
For $S = \mathbb{C}^n,$ the next statement was shown in \cite[Theorem~2]{Nemethi1990}.

\begin{corollary}\label{Corrolary31}
Under the assumption of Theorem~\ref{Theorem32}, we have
\begin{eqnarray*}
B(f|_S)  &\subset& \Sigma_{\infty}(f|_S) \cup K_0(f|_S) \cup \{0\}.
\end{eqnarray*}
Moreover, if the polynomial $f  \colon \mathbb{C}^n \rightarrow \mathbb{C}$ is convenient, then $B(f|_S) = K_0(f|_S).$
\end{corollary}

\begin{proof}
This is an immediate consequence of Theorems~\ref{Theorem31}~and~\ref{Theorem32} and the fact that the bifurcation set $B(f|_S)$ contains the set $K_0(f|_S)$ of critical values of the restriction $f|_S.$
\end{proof}

\section{The stability of global monodromies} \label{Section4}

Recall that the (non-singular) algebraic set $S$ is given by
$$S := \{x \in \mathbb{C}^n \ : \ g_1(x)  = 0, \ldots, g_p(x)  = 0\}.$$
In what follows, let $f(t, x)$ be a polynomial in $x \in \mathbb{C}^n$ with coefficients which are smooth (i.e., $C^\infty$) complex valued functions of $t \in [0, 1].$ We will write $f_t(x) := f(t, x)$ and assume that for each $t \in [0, 1],$ the restriction $f_t|_S \colon S \rightarrow \mathbb{C}$ is dominant (i.e., the image set $f_t (S)$ is dense in $\mathbb{C}$). With these preparations, we have the following stability result, which generalizes \cite[Theorem~17]{Nemethi1992} and \cite[Theorem~1.1]{Phamts2010}.

\begin{theorem} \label{Theorem41}
Let the following conditions are satisfied: 
\begin{enumerate}
\item[(i)] The Newton polyhedron of ${f}_t$ is independent of $t;$
\item[(ii)] For each $t \in [0, 1],$ the restriction $f_t|_S$ is Newton non-degenerate at infinity.
\end{enumerate}   
Then the global monodromies of the $f_t|_S$ are all isomorphic.
\end{theorem}

The proof of Theorem~\ref{Theorem41} will be divided into several steps, which, for convenience, will be called lemmas.

\begin{lemma}[Boundedness of affine singularities]\label{Lemma41}
There exists a real number $r > 0$ such that 
\begin{eqnarray*}
K_0(f_t|_S) & \subset & D_r \quad \textrm{ for all } \quad t\in [0, 1].
\end{eqnarray*}
\end{lemma}

\begin{proof}
Suppose the lemma were false. Then by the Curve Selection Lemma at infinity (see \cite{Nemethi1992} or \cite{HaHV2017}), there exist analytic curves 
\begin{eqnarray*}
\phi \colon (0,\epsilon) \rightarrow \mathbb{C}^n, \quad 
t \colon (0,\epsilon) \rightarrow [0, 1], \quad \textrm{ and } \quad 
\lambda_j \colon (0,\epsilon) \rightarrow \mathbb{C}, j = 1, \ldots, p,
\end{eqnarray*}
such that
\begin{enumerate}
\item[(a1)] $\Vert \phi(s) \Vert\rightarrow \infty$ as $s\rightarrow 0$;
\item[(a2)] $t(s)\rightarrow t_0 \in [0, 1]$ as $s\rightarrow 0$;
\item[(a3)] $f_{t(s)}(\phi(s))\rightarrow \infty$ as $s\rightarrow 0$;
\item[(a4)] $g_j(\phi(s))=0$ for all $j=1,\ldots,p,$ and all $s\in (0,\epsilon)$;
\item[(a5)] $\nabla f_{t(s)}(\phi(s)) + \sum_{j=1}^p \lambda_j(s)\nabla g_j(\phi(s))= 0\textrm{ for } s\in (0,\epsilon).$
\end{enumerate}

Put $I:=\{ i \ : \ \phi_i \not\equiv 0 \}$. By the condition (a1), $I \neq \emptyset$. For $i\in I$, we can write the curve $\phi_i$ in terms of parameter, say
\begin{eqnarray*}
\phi_i(s) = x^0_i s^{q_i} + \textrm{higher-order terms in }s,
\end{eqnarray*} 
where  $x_i^0 \neq 0$ and $q_i \in \mathbb{Q}.$ Observe that $\min_{i \in I} q_i <0$ because of the condition (a1).

Recall from our assumptions that the Newton polyhedron $\Gamma(f_{t})$ of $f_{t}$ does not depend on $t.$ By the condition~(a3), $\mathbb{R}^I \cap \Gamma(f_{t})\neq \emptyset.$ Let $d_0$ be the minimal value of the linear function $\sum_{i \in I}\alpha_i q_i$ on $ \mathbb{R}^I \cap \Gamma(f_{t})$ and $\Delta_0$ be the maximal face of $\mathbb{R}^I \cap \Gamma(f_{t})$ where this linear function takes its minimum value. We can write
\begin{eqnarray*}
f_{t(s)}(\phi(s)) &=& f_{t_0,\Delta_0}(x^0) s^{d_0} + \textrm{ higher-order terms in }s,\\
\frac{\partial f_{t(s)}}{\partial x_i}(\phi(s)) &=& \frac{\partial f_{t_0,\Delta_0}}{\partial x_i}(x^0)s^{d_0 - q_i} + \textrm{ higher-order terms in } s
\quad \textrm{ for all } i \in I, 
\end{eqnarray*}
where $x^0 := (x^0_1, \ldots, x^0_n)$ with $x^0_i = 1$ for $i \not \in I$ and $f_{t_0,\Delta_0}$ denotes the face function corresponding to $f_{t_0}$ and $\Delta_0.$ By the condition (a3), $d_0 < 0.$ Furthermore, for $i \notin I$, the function $f_{t_0,\Delta_0}$ does not depend on the variable $x_i$, and so
\begin{eqnarray}\label{EQ06}
\frac{\partial f_{t_0,\Delta_0}}{\partial x_i}(x^0) & = & 0 \quad \textrm{ for all } \quad i \not \in I.
\end{eqnarray}

Put $J := \{ j\in \{1,\ldots,p\} \ : \ \lambda_j \not\equiv 0 \}.$ If $J = \emptyset$, then from the condition (a5) we deduce for all $i \in I$ that
$\frac{\partial f_{t(s)}}{\partial x_i}(\phi(s)) = 0,$ and hence that $\frac{\partial f_{t_0,\Delta_0}}{\partial x_i}(x^0) = 0.$ It turns out from \eqref{EQ06}, the Euler relation, and the inequality $d_0 < 0$ that $f_{t_0,\Delta_0}(x^0)  = 0,$ which contradicts the non-degeneracy condition. Therefore, $J \neq \emptyset.$ For $j\in J$, we can write
$$\lambda_j(s) = c_js^{m_j}+\textrm{ higher-order terms in }s,$$
where $c_j\neq 0$ and $m_j \in \mathbb{Q}.$

Put $J_1 := \{ j\in J \ : \ g_j|_{\mathbb{C}^I} \not\equiv 0\}.$ If $J_1 = \emptyset,$ then
\begin{eqnarray*}
\frac{\partial g_j}{\partial x_i}(\phi(s)) \ \equiv\ 0 \quad \textrm{ for all } \quad i \in I \textrm{ and  all } j \in J.
\end{eqnarray*}
We deduce from the condition (a5) that
\begin{eqnarray*}
\frac{\partial f_{t(s)}}{\partial x_i}(\phi(s)) \ \equiv\ 0 \quad \textrm{ for all } \quad i \in I.
\end{eqnarray*}
Consequently, 
\begin{eqnarray*}
\frac{\partial f_{t_0,\Delta_0}}{\partial x_i}(x^0)\ =\ 0 \quad \textrm{ for all } \quad i \in I.
\end{eqnarray*}
It follows from \eqref{EQ06}, the Euler relation, and the inequality $d_0 < 0$ that $f_{t_0,\Delta_0}(x^0)  = 0,$ 
which contradicts the non-degeneracy condition. Hence $J_1\neq \emptyset$. For each $j\in J_1$, let $d_j$ be the minimal value of the linear function $\sum_{i \in I}\alpha_i q_i$ on $ \mathbb{R}^I \cap \Gamma(g_j)$ and $\Delta_j$ be the maximal face of $\mathbb{R}^I \cap \Gamma(g_j)$ where this linear function takes its minimum value. We can write
$$g_j(\phi(s))\ =\ g_{j,\Delta_j}(x^0) s^{d_j} + \textrm{ higher-order terms in }s,$$
where $g_{j,\Delta_j}$ is the face function associated with $g_j$ and $\Delta_j.$ By the condition (a4), then
\begin{eqnarray}\label{EQ07}
 g_{j,\Delta_j}(x^0) & = & 0 \quad \textrm{ for all } \quad j \in J_1.
\end{eqnarray}
On the other hand, for $i\in I$ and $j\in J_1$,
$$\frac{\partial g_j}{\partial x_i}(\phi(s)) \ =\ \frac{\partial g_{j,\Delta_j}}{\partial x_i} (x^0) s^{d_j - q_i} + \textrm{ higher-order terms in }s.$$
For $i\notin I$ and $j\in J_1$, the function $g_{j,\Delta_j}$ does not depend on the variable $x_i$, and hence,
\begin{eqnarray}\label{EQ08}
 \frac{\partial g_{j,\Delta_j}}{\partial x_i}(x_0)\ =\ 0.
\end{eqnarray}
The condition (a5) implies that for all $i\in I$,
\begin{eqnarray}\label{EQ09}
\frac{\partial f_{t_0,\Delta_0}}{\partial x_i}(x^0)s^{d_0-q_i} + \cdots + \sum_{j \in J_2} \overline{c_j} \frac{\partial g_{j,\Delta_j}}{\partial x_i}(x^0) s^{\ell-q_i} + \cdots\ =\ 0,
\end{eqnarray}
where $\ell:= \min_{j\in J_1}(m_j+d_j)$, $J_2:=\{j \in J_1 \ : \ \ell\ =\ m_j+d_j\}$ and the dots stand for the higher-order terms in $s$.
 There are three cases to be considered.

\subsubsection*{Case 1:} $\ell\ >\ d_0$. By \eqref{EQ06} and \eqref{EQ09}, we have
$$\frac{\partial f_{t_0,\Delta_0}}{\partial x_i}(x^0)\ =\ 0 \quad \textrm{ for }\quad i=1,\ldots,n.$$
This, together with the Euler relation, implies that
$$d_0 f_{t_0,\Delta_0}\ =\ 0.$$
Hence, $f_{t_0,\Delta_0}(x^0) = 0$ because of $d_0 < 0.$ This contradicts the non-degeneracy condition.

\subsubsection*{Case 2:} $\ell\ =\ d_0$. We deduce from \eqref{EQ06}, \eqref{EQ08} and \eqref{EQ09} that
$$\frac{\partial f_{t_0,\Delta_0}}{\partial x_i}(x^0) + \sum_{j\in J_2} \overline{c_j} \frac{\partial g_{j,\Delta_j}}{\partial x_i}(x^0)\ =\ 0 \quad\textrm{  for  }\quad i=1,\ldots,n.$$
Consequently,
\begin{eqnarray*}
0 
& = & \sum_{i = 1}^n q_i x^0_i \frac{\partial f_{t_0,\Delta_0}}{\partial x_i}(x^0)+\sum_{i = 1}^n \sum_{j\in J_2} \overline{c_j} q_i x^0_i \frac{\partial g_{j,\Delta_j}}{\partial x_i}(x^0) \\
& = & \sum_{i = 1}^n  q_i x^0_i \frac{\partial f_{t_0,\Delta_0}}{\partial x_i}(x^0) + \sum_{j\in J_2} \overline{c_j} \sum_{i = 1}^n q_i x^0_i \frac{\partial g_{j,\Delta_j}}{\partial x_i}(x^0) \\
& = & d_0 f_{t_0,\Delta_0}(x^0)+ \sum_{j\in J_2} \overline{c_j} d_j g_{j,\Delta_j}(x^0) \\
& = & d_0 f_{t_0,\Delta_0}(x^0),
\end{eqnarray*}
where the last equation follows from \eqref{EQ07}. Since $d_0 < 0,$ we get $f_{t_0, \Delta_0}(x^0) = 0,$ which contradicts the non-degeneracy condition.

\subsubsection*{Case 3:} $\ell\ <\ d_0$. By \eqref{EQ08} and \eqref{EQ09}, we obtain
$$ \sum_{j\in J_2} \overline{c_j} \frac{\partial g_{j,\Delta_j}}{\partial x_i}(x^0)= 0 \quad\textrm{  for  }\quad i=1,\ldots,n.$$
This fact and \eqref{EQ07} combined give a contradiction with the non-degeneracy condition.
\end{proof}

\begin{lemma}[Boundedness of singularities at infinity]\label{Lemma42}
There exists a real number $r > 0$ such that
\begin{eqnarray*}
\Sigma_\infty(f_{t}|_S) & \subset & D_r \quad \textrm{ for all }\quad t \in [0, 1].
\end{eqnarray*}
\end{lemma}

\begin{proof}
Suppose the assertion of the lemma is false. By the Curve Selection Lemma at infinity (see \cite{Nemethi1992} or \cite{HaHV2017}), we can find a nonempty set $I \subset \{ 1, \ldots, n \}$ with ${f_{t}}|_{\mathbb{C}^I} \not\equiv 0,$ a (possibly empty) set $J \subset \{j \in \{1, \ldots, p\} \ : \ g_j|_{\mathbb{C}^I} \not \equiv 0\},$ a vector $ q \in \mathbb{R}^n $ with $ \min_{ i \in I} q_i < 0$ and $ 0 = d( q, \Gamma( {f_{t}}|_{{\mathbb{C}^I}}))$, and analytic curves
\begin{eqnarray*}
\phi \colon (0,\epsilon) \rightarrow (\mathbb{C}^*)^I, \quad 
t \colon (0,\epsilon) \rightarrow [0, 1], \quad \textrm{ and } \quad 
\lambda_j \colon (0,\epsilon) \rightarrow \mathbb{C}, j \in J,
\end{eqnarray*}
such that the following conditions hold
\begin{enumerate}
\item[(a1)] $\|\phi(s)\| \rightarrow \infty$ as $s \rightarrow 0$;
\item[(a2)] $t(s) \rightarrow t_0 \in [0, 1]$ as $s \rightarrow 0$;
\item[(a3)] $f_{t(s), \Delta_0}( \phi(s)) \rightarrow \infty$ as $s \rightarrow 0;$
\item[(a4)] $g_{ j, \Delta_j}(\phi(s)) = 0$ for all $j \in J$ and all $s \in ( 0, \epsilon);$
\item[(a5)] $\nabla f_{t(s), \Delta_0}( \phi(s)) + \sum_{j \in J} \lambda_j(s) \nabla g_{ j, \Delta_j}( \phi(s)) = 0\textrm{ for all } s \in ( 0, \epsilon)$,
\end{enumerate}
where $\Delta_0 := \Delta(q, \Gamma({f_{t}}|_{\mathbb{C}^I})$ and $\Delta_j := \Delta( q, \Gamma({g_j}|_{{\mathbb{C}^I}}))$ for $j\in J.$

For $i \in I,$ we can write the curve $\phi_i$ in terms of parameter, say
\begin{eqnarray*}
\phi_i(s) = x^0_i s^{ q'_i} + \textrm{higher-order terms in }s,
\end{eqnarray*} 
where  $x_i^0 \neq 0$ and $q'_i \in \mathbb{Q}.$ Observe that $\min_{i \in I} q'_i < 0$ because of the condition (a1). 

Let $d_0$ be the minimal value of the linear function $\sum_{i \in I}\alpha_i q'_i$ on $\mathbb{R}^I\cap \Delta_0$ ($= \Delta_0$) and $\Delta_0'$ be the maximal face of $\Delta_0$ where this linear function takes its minimum value. As the Newton polyhedron $\Gamma(f_{t})$ of $f_{t}$ does not depend on $t$, we can write
\begin{eqnarray*}
f_{t(s),\Delta_0}(\phi(s)) &=& f_{ t_0,{\Delta_0'}}(x^0) s^{d_0} + \textrm{ higher-order terms in }s,\\
\frac{\partial f_{t(s),\Delta_0}(\phi(s))}{\partial x_i}  &=& \frac{\partial f_{ t_0,{\Delta_0'}}}{\partial x_i}(x^0)s^{d_0 - q'_i} + \textrm{ higher-order terms in } s \quad \textrm{ for } i \in I,
\end{eqnarray*}
where $x^0 := (x^0_1, \ldots, x^0_n)$ with $x^0_i = 1$ for $i \not \in I.$  By the condition (a3), $d_0 < 0.$ Furthermore, for $i \notin I$, the function $f_{ t_0,{\Delta_0'}}$ does not depend on the variable $x_i$, and so
\begin{eqnarray}\label{EQ10}
\frac{\partial f_{t_0, {\Delta_0'}} } {\partial x_i}(x^0) & = & 0 \quad \textrm{ for all } \quad i \not \in I.
\end{eqnarray}

Put $J_1: = \{ j \in J \ : \ \lambda_j \not\equiv 0 \}$. If $J_1 = \emptyset $, then the condition (a5) implies that
\begin{eqnarray*}
\frac{\partial f_{t(s), \Delta_0}}{\partial x_i}(\phi(s)) \ \equiv\ 0 \quad \textrm{ for all } \quad i \in I.
\end{eqnarray*}
Consequently, 
\begin{eqnarray*}
\frac{ \partial f_{ t_0,{\Delta_0'}} }{ \partial x_i}( x^0) & = & 0 \quad \textrm{ for all } \quad  i \in I. 
\end{eqnarray*}
Since $d_0 < 0$, it follows from \eqref{EQ10} and the Euler relation that
\begin{eqnarray*}
f_{ t_0, {\Delta_0'}}(x^0) & = & 0,  
\end{eqnarray*}
which contradicts the non-degeneracy condition.
Thus, $J_1 \neq \emptyset$. For $j\in J_1$, we can write
\begin{eqnarray*}
\lambda_j(s) &=& c_j s^{m_j} + \textrm{ higher-order terms in }s,
\end{eqnarray*}
where $c_j \neq 0$ and $m_j \in \mathbb{Q}.$

For each $j \in J_1$, let $d_j$ be the minimal value of the linear function $\sum_{i \in I}\alpha_i q'_i$ on $\mathbb{R}^I \cap \Delta_j$ $(=\Delta_j)$ and $\Delta_j'$ be the maximal face of $\Delta_j$ where this linear function takes its minimum value. We can write
\begin{eqnarray*}
g_{j, \Delta_j}( \phi(s)) &=& g_{j, \Delta_j'}( x^0) s^{d_j} + \textrm{ higher-order terms in }s.
\end{eqnarray*}
By the condition (a4), we have
\begin{eqnarray}\label{EQ11}
 g_{j, \Delta_j'}( x^0) & = & 0 \quad \textrm{ for all } \quad j \in J_1.
\end{eqnarray}
On the other hand, a direct calculation shows that for $i\in I$ and $j\in J_1$,
\begin{eqnarray*}
\frac{\partial g_j}{\partial x_i}(\phi(s)) &=& \frac{\partial g_{j,\Delta'_j}}{\partial x_i} (x^0) s^{d_j - q'_i} + \textrm{ higher-order terms in }s.
\end{eqnarray*}
For $i\notin I$ and $j \in J_1$, the function $g_{j,{\Delta_j'}}$ does not depend on the variable $x_i$, and so
\begin{eqnarray}\label{EQ12}
\frac{\partial g_{ j, {\Delta_j'}} } {\partial x_i}(x^0) & = & 0 \quad \textrm{ for all } \quad i \notin I \quad \textrm{ and } \quad j \in J_1.
\end{eqnarray}

The condition (a5) implies that for all $i\in I$,
\begin{eqnarray}\label{EQ13}
\frac{\partial f_{t_0, {\Delta_0'}} } { \partial x_i} ( x^0)s^{ d_0-q'_i} + \cdots + \sum_{ j \in J_2} \overline{c_j} \frac{ \partial g_{ j,\Delta_j'}} {\partial x_i}( x^0) s^{ \ell - q'_i} + \cdots \ =\ 0,
\end{eqnarray}
where $\ell := \min_{j \in J_1}( m_j + d_j)$, $J_2:=\{ j \in J_1 \ : \ \ell\ =\ m_j + d_j\}$ and the dots stand for the higher-order terms in $s$.
There are three cases to be considered.

\subsubsection*{Case 1:} $\ell\ >\ d_0$. By \eqref{EQ10} and \eqref{EQ13}, we have
$$\frac{\partial f_{ t_0,{\Delta_0'}}}{\partial x_i}(x^0)\ =\ 0 \quad \textrm{ for }\quad i=1,\ldots,n.$$
This, together with the Euler relation, implies that
$$d_0f_{ t_0,{\Delta_0'}}(x^0)\ =\ 0.$$
Hence, $f_{ t_0,{\Delta_0'}}(x^0) = 0$ because of $d_0 < 0.$ This contradicts the non-degeneracy condition.

\subsubsection*{Case 2:} $\ell\ =\ d_0$. We deduce from \eqref{EQ10}, \eqref{EQ12} and \eqref{EQ13} that
$$ \frac{\partial f_{ t_0,{\Delta_0'}}}{\partial x_i}(x^0) + \sum_{j\in J_2} \overline{c_j} \frac{\partial g_{j,\Delta'_j}}{\partial x_i}(x^0)\ =\ 0 \quad\textrm{  for  }\quad i=1,\ldots,n.$$
Consequently,
\begin{eqnarray*}
0 
& = & \sum_{i = 1}^n q'_i x^0_i \frac{\partial f_{ t_0,{\Delta_0'}}}{\partial x_i}(x^0) + \sum_{i = 1}^n \sum_{j\in J_2} \overline{c_j} q'_i x^0_i \frac{\partial g_{j,\Delta'_j}}{\partial x_i}(x^0) \\
& = & \sum_{i = 1}^n  q'_i x^0_i \frac{\partial f_{ t_0,{\Delta_0'}}}{\partial x_i}(x^0) + \sum_{j\in J_2} \overline{c_j} \sum_{i = 1}^n q'_i x^0_i \frac{\partial g_{j,\Delta'_j}}{\partial x_i}(x^0) \\
& = & d_0f_{ t_0,{\Delta_0'}}(x^0)+ \sum_{j\in J_2} \overline{c_j} d_j g_{j,\Delta'_j}(x^0) \\
& = & d_0f_{ t_0,{\Delta_0'}}(x^0),
\end{eqnarray*}
where the last equation follows from \eqref{EQ11}. Since $d_0 < 0,$ we get $f_{ t_0,{\Delta_0'}}(x^0) = 0,$ which contradicts the non-degeneracy condition.

\subsubsection*{Case 3:} $\ell\ <\ d_0$. By \eqref{EQ12} and \eqref{EQ13}, we obtain
$$ \sum_{j\in J_2} \overline{c_j} \frac{\partial g_{j,\Delta'_j}}{\partial x_i}(x^0)= 0 \quad\textrm{  for  }\quad i=1,\ldots,n.$$
This fact, together with \eqref{EQ11}, gives a contradiction with the non-degeneracy condition.
\end{proof}

\begin{lemma}[Transversality in the neighbourhood of infinity] \label{Lemma43}
Let $r$ be a positive real number such that the conclusions of Lemmas~\ref{Lemma41}~and~\ref{Lemma42} are fulfilled. Then there exists a real number $R_0 > 0$ such that for all $t \in [0,1],$ all $R \geqslant R_0$ and all $c \in \mathbb{S}^1_r,$ we have the fiber $({f}_t|_S)^{-1}(c)$ intersects transversally with the sphere $\mathbb{S}^{2n - 1}_R.$
\end{lemma}

\begin{proof}
If the assertion is not true, then by the Curve Selection Lemma at infinity (see \cite{Nemethi1992} or \cite{HaHV2017}), there exist $t_0 \in [0,1]$, $c \in  \mathbb{S}^1_r$ and analytic curves
\begin{eqnarray*}
\phi \colon ( 0, \epsilon ) \rightarrow \mathbb{C}^n, \quad 
t \colon ( 0, \epsilon ) \rightarrow [0, 1], \quad \textrm{ and } \quad 
\lambda_j \colon ( 0, \epsilon) \rightarrow \mathbb{C}, j = 0, 1, \ldots, p + 1,
\end{eqnarray*}
satisfying the following conditions
\begin{enumerate}
\item[(a1)] $\Vert \phi(s) \Vert \rightarrow \infty$ as $s \rightarrow 0;$
\item[(a2)] $t(s) \rightarrow t_0$ as $s \rightarrow 0;$
\item[(a3)] $f_{t(s)}(\phi(s)) \rightarrow c$ as $s \rightarrow 0;$
\item[(a4)] $g_j( \phi(s)) = 0$ for all $j = 1, \ldots, p,$ and all $s\in (0,\epsilon);$
\item[(a5)] $\lambda_0(s) \nabla f_{t(s)} (\phi(s)) + \sum_{j = 1}^p \lambda_j(s) \nabla g_j( \phi(s)) = \lambda_{p+1}(s) {\phi(s)} \textrm{ for all } s\in (0, \epsilon)$.
\end{enumerate}

Put $I := \{ i \ : \ \phi_i \not\equiv 0 \}$. By the condition (a1), $I \neq  \emptyset$. For $i \in I$, we can write the curve $\phi_i$ in terms of parameter, say
\begin{eqnarray*}
\phi_i(s) &=& \ x^0_i s^{q_i} + \textrm{higher-order terms in }s,
\end{eqnarray*}
where $x_i^0 \neq 0$ and $q_i \in \mathbb{Q}.$ Observe that $\min_{i \in I} q_i < 0,$ because of the condition (a1).

By the condition (a3) and the fact that $|c| = r > 0,$ we have $\mathbb{R}^I \cap \Gamma(f_{t}) \neq \emptyset.$ Let $d_0$ be the minimal value of the linear function $\sum_{i \in I}\alpha_i q_i$ on $ \mathbb{R}^I \cap \Gamma(f_{t})$ and $\Delta_0$ be the maximal face of $\mathbb{R}^I \cap \Gamma(f_{t})$ where this linear function takes its minimum value. As the Newton polyhedron $\Gamma(f_{t})$ of $f_{t}$ does not depend on $t$, we can write
\begin{eqnarray*}
f_{t(s)}(\phi(s)) &=& f_{ t_0,{\Delta_0}}(x^0) s^{d_0} + \textrm{ higher-order terms in }s,\\
\frac{\partial f_{t(s)}}{\partial x_i} (\phi(s)) &=& \frac{\partial f_{ t_0,{\Delta_0}}}{\partial x_i}(x^0)s^{d_0 - q_i} + \textrm{ higher-order terms in } s \quad \textrm{ for } i \in I,
\end{eqnarray*}
where $x^0 := (x^0_1, \ldots, x^0_n)$ with $x^0_i = 1$ for $i \not \in I.$ 
The condition (a3) and the fact that $|c| = r > 0$ together imply that
\begin{eqnarray}\label{EQ14}
d_0 \ \leqslant \  0 \quad \textrm{ and }\quad d_0 f_{ t_0,{\Delta_0}}(x^0)\ =\ 0.
\end{eqnarray}
Furthermore, for $i \notin I$, the function $f_{ t_0,{\Delta_0}}$ does not depend on the variable $x_i$, and so
\begin{eqnarray}\label{EQ15}
\frac{\partial f_{ t_0,{\Delta_0}}}{\partial x_i}(x^0) & = & 0 \quad \textrm{ for all } \quad i \not \in I.
\end{eqnarray}
 
On the other hand, we deduce from the condition (a5), Lemmas~\ref{Lemma31}~and~\ref{Lemma41} that 
$\lambda_0 \not\equiv 0$ and $\lambda_{p + 1} \not\equiv 0$ (perhaps reducing $\epsilon$). 
Replacing $\lambda_j$ by $\frac{\lambda_j}{\lambda_0}$ if necessary, we may assume that $\lambda_0 \equiv 1$. Put $J:= \{ j\in \{1, \ldots, p \} \ : \ \lambda_j \not\equiv 0 \}$. For $j \in J \cup \{p+1\},$ we can write 
$$\lambda_j(s) \ =\ c_j s^{m_j} +  \textrm{higher-order terms in }s ,$$
where $c_j \neq 0$ and $m_j\in \mathbb{Q}.$ 

Put $J_1:= \{ j\in J \ : \ {g_j}|_{\mathbb{C}^I} \not\equiv 0\}$. There are two cases to be considered.

\subsubsection*{Case 1} $J_1 = \emptyset.$ 
The condition (a5) implies that for $i \in I,$ 
\begin{eqnarray*}
\frac{ \partial f_{t_0, \Delta_0}} {\partial x_i}(x^0) s^{d_0 - q_i} + \cdots &=& \overline{c_{p+1}} \overline{x^0_i} s^{m_{p+1} + q_i} +  \cdots,
\end{eqnarray*}
where the dots stand for higher-order terms in $s$. Clearly, $d_0 - q_i \leqslant m_{p+1} + q_i$ for all $ i \in I$, and so
\begin{eqnarray}\label{EQ16}
d_0 - m_{p+1} & \leqslant & 2\min_{i \in I} q_i \ <\ 0.
\end{eqnarray}

Put $I_1 := \{i \in I \ : \ d_0 - q_i = m_{p+1} + q_i\}.$ Obverse that $i \in I \setminus I_1$ if, and only if, 
\begin{eqnarray}\label{EQ17}
\frac{\partial f_{t_0,\Delta_0}} {\partial x_i}(x^0) &=& 0, 
\end{eqnarray}
and in this case $d_0 - q_i < m_{p+1} + q_i.$

If $I_1= \emptyset,$ then we get from \eqref{EQ15} and \eqref{EQ17} 
\begin{eqnarray*}
\frac{\partial f_{t_0,\Delta_0}} {\partial x_i}(x^0) &=& 0 \quad \textrm{ for all }\quad i=1,\ldots, n.
\end{eqnarray*}
Hence, \eqref{EQ14} and the non-degeneracy condition together imply that $d_0 = 0.$ By the condition (a3), $c = f_{t_0,\Delta_0}(x^0) \in \Sigma_\infty(f_{t_0}|_S),$ which contradicts our assumption.

If $I_1 \neq \emptyset$, then for all $i \in I_1$,
\begin{eqnarray*}
\frac{\partial f_{t_0,\Delta_0}} {\partial x_i}(x^0)  = \overline{c_{p+1}} \overline{x^0_i} \quad \textrm{ and }\quad d_0 - m_{p+1} = 2 q_i. 
\end{eqnarray*}
Hence, the Euler relation, \eqref{EQ15} and \eqref{EQ17} together imply that
\begin{eqnarray*}
d_0 f_{t_0,\Delta_0} &=& \sum_{i=1}^n \frac{\partial f_{t_0,\Delta_0}}{\partial x_i}(x^0)x^0_i q_i \ =\ \sum_{i\in I_1} \frac{\partial f_{t_0,\Delta_0}}{\partial x_i}(x^0)x^0_i q_i = \sum_{i \in I_1} |x^0_i|^2 \frac{d_0 - m_{p+1}}{2} \overline{c_{p+1}} \ne 0,
\end{eqnarray*}
where the last inequality follows from \eqref{EQ16}. This, combined with \eqref{EQ14}, implies a contradiction.
  
\subsubsection*{Case 2} $J_1 \neq \emptyset.$ 
For each $j\in J_1$, let $d_j$ be the minimal value of the linear function $\sum_{i \in I}\alpha_i q_i$ on $ \mathbb{R}^I \cap \Gamma(g_j)$ and $\Delta_j$ be the maximal face of $\mathbb{R}^I \cap \Gamma(g_j)$ where this linear function takes its minimum value. We can write
$$g_j(\phi(s))\ =\ g_{j,\Delta_j}(x^0) s^{d_j} + \textrm{ higher-order terms in }s.$$
By the condition (a4), then
\begin{eqnarray}\label{EQ18}
 g_{j,\Delta_j}(x^0) & = & 0 \quad \textrm{ for all } \quad j \in J_1,
\end{eqnarray}
On the other hand, for $i\in I$ and $j\in J_1$,
$$\frac{\partial g_j}{\partial x_i}(\phi(s)) \ =\ \frac{\partial g_{j,\Delta_j}}{\partial x_i} (x^0) s^{d_j - q_i} + \textrm{ higher-order terms in }s.$$
For $i\notin I$ and $j\in J_1$, the function $g_{j,\Delta_j}$ does not depend on the variable $x_i$, and hence,
\begin{eqnarray}\label{EQ19}
 \frac{\partial g_{j,\Delta_j}}{\partial x_i}(x_0)\ =\ 0.
\end{eqnarray}

Form the condition (a5), for $i\in I$ we have
\begin{eqnarray}\label{EQ20}
\frac{\partial f_{t_0,\Delta_0}}{\partial x_i}(x^0)s^{d_0-q_i}+\cdots+\sum_{j\in J_2} \overline{c_j} \frac{\partial g_{j,\Delta_j}}{\partial x_i}(x^0)s^{\ell-q_i}+\cdots 
&=& \overline{c_{p+1}} \overline{x^0_i}s^{m_{p+1}+q_i}+\cdots 
\end{eqnarray}
where $\ell:= \min_{j\in J_1}(m_j+d_j)$ and $J_2:=\{j \in J_1 \ : \ m_j+d_j=\ell\}$ and the dots stand for higher-order terms in $s$. There are three cases to be considered.

\subsubsection*{Case 2.1} $\ell > d_0.$
The same argument as in Case 1 yields a contradiction.

\subsubsection*{Case 2.2} $\ell = d_0.$
From \eqref{EQ20} we have $d_0 - q_i \leqslant m_{p + 1} + q_i$ for all $i \in I.$ Therefore
\begin{eqnarray*}
d_0 - m_{p + 1} & \leqslant  & 2 \min_{i \in I}  q_i \ < \ 0.
\end{eqnarray*}

Put $I_2 :=\{i \in I \ : \ d_0 - q_i = m_{p+1} + q_i\}$. Hence, $i \in I\setminus I_2$ if, and only if,
$$ \frac{\partial f_{t_0,\Delta_0}}{\partial x_i}(x^0) + \sum_{j\in J_2} \overline{c_j} \frac{\partial g_{j,\Delta_j}}{\partial x_i}(x^0)= 0,$$ and in this case $d_0-q_i < m_{p+1} + q_i.$

If $I_2 = \emptyset,$ then 
$$ \frac{\partial f_{t_0, \Delta_0}}{\partial x_i}(x^0) + \sum_{j\in J_2} \overline{c_j}  \frac{\partial g_{j,\Delta_j}}{\partial x_i}(x^0)\ =\ 0 \quad \textrm{ for all }\quad i = 1, \ldots, n.$$
Hence, the non-degeneracy condition, \eqref{EQ14} and \eqref{EQ18} together imply that $d_0 = 0.$ Consequently, by the condition (a2), $c  = f_{t_0, \Delta_0} (x^0) \in \Sigma_\infty(f_{t_0}|_S),$ which contradicts our assumption.

If $I_2\neq\emptyset$, then from \eqref{EQ20} we have for all $i \in I_2,$
\begin{eqnarray*}
\frac{\partial f_{t_0,\Delta_0}}{\partial x_i}(x^0) + \sum_{j\in J_2} \overline{c_j}   \frac{\partial g_{j,\Delta_j}}{\partial x_i}(x^0) & = &  \overline{c_{p+1}} \overline{x^0_i}, \\
d_0-m_{p+1} & = & 2q_i.
\end{eqnarray*}
This, together with the Euler relation, \eqref{EQ14}, \eqref{EQ15}, \eqref{EQ18} and \eqref{EQ19}, yields
\begin{eqnarray*} 
0 & = & d_0 f_{t_0,\Delta_0}(x^0)+ \sum_{j\in J_2} \overline{c_j} d_j g_{j,\Delta_j}(x^0)\\
& = &\sum_{i = 1}^n q_i x^0_i \frac{\partial f_{ t_0,\Delta_0}}{\partial x_i}(x^0) + \sum_{j\in J_2} \sum_{i = 1}^n \overline{c_j} q_i x^0_i \frac{\partial g_{j,\Delta_j}}{\partial x_i}(x^0) \\ 
& = & \sum_{i = 1}^n q_i x^0_i \left (\frac{\partial f_{ t_0,\Delta_0}}{\partial x_i}(x^0)+ \sum_{j\in J_2} \overline{c_j} \frac{\partial g_{j,\Delta_j}}{\partial x_i}(x^0) \right) \\
& = & \sum_{i \in I_2} q_i x^0_i \left(\frac{\partial f_{ t_0,\Delta_0}}{\partial x_i}(x^0)+ \sum_{j\in J_2} \overline{c_j}  \frac{\partial g_{j,\Delta_j}}{\partial x_i}(x^0) \right) \\
& = &  \sum_{i \in I_2} |x^0|_i^2 \frac{d_0-m_{p+1}}{2} \overline{c_{p + 1}} \ \ne \ 0,
\end{eqnarray*}
which is impossible.

\subsubsection*{Case 2.3} $\ell < d_0.$
From \eqref{EQ20} we have $\ell - q_i \leqslant  m_{p + 1} + q_i$ for all $i \in I.$ Therefore
\begin{eqnarray*}
\ell - m_{p + 1} & \leqslant  & 2 \min_{i \in I}  q_i \ < \ 0.
\end{eqnarray*}

Put $I_3 :=\{i \in I \ : \ \ell - q_i = m_{p+1} + q_i\}$. Hence, $i \in I\setminus I_3$ if, and only if,
$$\sum_{j\in J_2} \overline{c_j} \frac{\partial g_{j,\Delta_j}}{\partial x_i}(x^0)= 0,$$ and in this case $\ell - q_i < m_{p+1} + q_i.$

If $I_3 = \emptyset,$ then 
$$ \sum_{j\in J_2} \overline{c_j}  \frac{\partial g_{j,\Delta_j}}{\partial x_i}(x^0)\ =\ 0 \quad \textrm{ for all } \quad i = 1, \ldots, n,$$
which, together with \eqref{EQ18}, leads to a contradiction with the non-degeneracy condition.

If $I_3 \ne \emptyset,$ then from \eqref{EQ20} we have  for all $ i\in I_3,$
\begin{eqnarray*}
\sum_{j\in J_2} \overline{c_j}   \frac{\partial g_{j,\Delta_j}}{\partial x_i}(x^0) &  = & \overline{c_{p+1}} \overline{x^0_i}, \\
\ell-m_{p+1} & = & 2q_i.
\end{eqnarray*}
This, together with the Euler relation and \eqref{EQ18}, yields 
\begin{eqnarray*} 
0 & = &  \sum_{j\in J_2} \overline{c_j} d_j g_{j,\Delta_j}(x^0)\\
& = & \sum_{j\in J_2} \sum_{i = 1}^n \overline{c_j} q_i x^0_i \frac{\partial g_{j,\Delta_j}}{\partial x_i}(x^0) \\ 
& = & \sum_{i = 1}^n \sum_{j\in J_2} \overline{c_j} q_i x^0_i \frac{\partial g_{j,\Delta_j}}{\partial x_i}(x^0) \\ 
& = & \sum_{i \in I_3} q_i x^0_i \left(\sum_{j\in J_2} \overline{c_j} \frac{\partial g_{j,\Delta_j}}{\partial x_i}(x^0) \right) \\
& = &  \sum_{i \in I_3} |x^0|_i^2 \frac{\ell - m_{p+1}}{2} \overline{c_{p + 1}} \ \ne \ 0,
\end{eqnarray*}
which is impossible.
\end{proof}

We now can complete the proof of the Theorem~\ref{Theorem41}.

\begin{proof}[Proof of Theorem~\ref{Theorem41}]
Let $r$ and $R_0$ be the positive real numbers such that the conclusions of Lemmas~\ref{Lemma31}, \ref{Lemma41}, \ref{Lemma42} and \ref{Lemma43} are fulfilled. By Corollary~\ref{Corrolary31}, then $B(f_t|_S) \subset D_r$ for all $t \in [0, 1].$ Furthermore, for all $(t, x) \in X := \{(t, x) \in [0, 1] \times S \ : \ f(t, x) \in \mathbb{S}_r^1, \|x\| \geqslant R_0\},$ the vectors 
$\nabla f_t(x), \nabla g_1(x), \ldots, \nabla g_p(x),$ and $\overline{x}$ are $\mathbb{C}$-linearly independent. Therefore, we can find a smooth map $\mathbf{v}_1
 \colon X \rightarrow \mathbb{C}^n, (t, x) \mapsto \mathbf{v}_1(t, x),$ satisfying the following conditions
\begin{itemize}
\item[(a1)] $\langle \mathbf{v}_1(t, x), {\nabla {f_{t}}(x)}\rangle = - \frac{\partial f_{t}}{\partial t}(x);$

\item[(a2)] $\langle \mathbf{v}_1(t, x), {\nabla {g_j}(x)} \rangle = 0$ for $j = 1,\ldots, p;$

\item[(a3)] $\langle \mathbf{v}_1(t, x), x \rangle = 0.$ 
\end{itemize}

We take arbitrary (but fixed) $\epsilon > 0.$ Since $\mathbb{S}_r^1 \cap K_0(f_t|_S) = \emptyset$ for all $t \in [0, 1],$ the vectors $\nabla f_t(x), \nabla g_1(x), \ldots, \nabla g_p(x)$ are $\mathbb{C}$-linearly independent for all $(t, x) \in Y := \{(t, x) \in [0, 1] \times S \ : \ f(t, x) \in \mathbb{S}_r^1, \|x\| \leqslant R_0 + \epsilon\}.$ Consequently, there exists a smooth map $\mathbf{v}_2  \colon Y \rightarrow \mathbb{C}^n, (t, x) \mapsto \mathbf{v}_2(t, x),$ such that the following conditions hold
\begin{itemize}
\item[(a4)] $\langle \mathbf{v}_2(t, x), {\nabla {f_{t}}(x)}\rangle = - \frac{\partial f_{t}}{\partial t}(x);$

\item[(a5)] $\langle \mathbf{v}_2(t, x), {\nabla {g_j}(x)} \rangle = 0$ for $j = 1,\ldots, p.$
\end{itemize}

Next, by patching the maps $\mathbf{v}_1$ and $\mathbf{v}_2$ together using a smooth partition of unity, we get a smooth map 
$$\mathbf{v} \colon  \{(t, x) \in [0, 1] \times S \ : \ f(t, x) \in \mathbb{S}_r^1 \} \rightarrow \mathbb{C}^n, \quad (t, x) \mapsto \mathbf{v}(t, x),$$ 
such that the following conditions hold:
\begin{itemize}
\item[(a6)] $\langle \mathbf{v}(t, x), {\nabla {f_{t}}(x)}\rangle = - \frac{\partial f_{t}}{\partial t}(x);$

\item[(a7)] $\langle \mathbf{v}(t, x), {\nabla {g_j}(x)} \rangle = 0$ for $j = 1,\ldots, p;$

\item[(a8)] $\langle \mathbf{v}(t, x), x \rangle = 0$ provided that $\|x\| \geqslant  R_0 + \epsilon.$
\end{itemize}

Finally we can check that for each $x \in f_0^{-1}(\mathbb{S}_r^1) \cap S,$ there exists a unique $C^\infty$-map $\varphi \colon [0, 1] \rightarrow \mathbb{C}^n$ such that
\begin{eqnarray*}
\varphi'(t) &=& \mathbf{v}(t, \varphi(t)), \quad \varphi(0) \ = \  x.
\end{eqnarray*}
Moreover, for each $t \in [0, 1],$ the map 
$$\Phi_t \colon f_0^{-1}(\mathbb{S}_r^1) \cap S \rightarrow f_t^{-1}(\mathbb{S}_r^1) \cap S, \quad x \mapsto \varphi(t),$$ 
is well-defined and is a $C^\infty$-diffeomorphism, which makes the following diagram commutes
\begin{equation*}
\CD f_0^{-1}(\mathbb{S}_r^1) \cap S @> \Phi_t >> f_t^{-1}(\mathbb{S}_r^1) \cap S  \\
@V f_0 VV @V f_t VV\\
\mathbb{S}_r^1 @> \mathrm{id} >> \mathbb{S}_r^1
\endCD
\end{equation*}
where $\mathrm{id} $ denotes the identity map.
\end{proof}


\end{document}